\documentclass[11pt]{amsart}
\usepackage{mathpazo}
\usepackage[utf8]{inputenc}
\usepackage[T1]{fontenc}
\usepackage{amssymb}
\usepackage{xcolor}
\usepackage{mathrsfs}

\usepackage{hyperref}

\hypersetup{
colorlinks=true,
linkcolor=blue,
urlcolor=blue,
citecolor=blue,}

\def \H{{\mathcal H}}
\def \K {{\mathcal K}}

\def \I {{\mathbf I}}
\def \t^0 {{\mathbf t^0}}
\def \C {{\mathbb C}}
\def \A {{\mathbf{A}}}
\def \X {{\mathbf X}}
\def \M {{\mathbb M}}
\def \tr {{\operatorname{tr}}}
\def \Re {{\operatorname{Re}}}

\newtheorem{theorem}{Theorem}[section]%[subsection]
\newtheorem{lemma}[theorem]{Lemma}
\theoremstyle{definition}

\newtheorem{corollary}[theorem]{Corollary}
\newtheorem{proposition}[theorem]{Proposition}

\numberwithin{equation}{section}

\newcommand{\beqa}{\begin{eqnarray*}}
\newcommand{\eeqa}{\end{eqnarray*}}
\newcommand{\beqn}{\begin{eqnarray}}
\newcommand{\eeqn}{\end{eqnarray}}

\newcommand{\N}{I\!\!N}

\newcommand{\ph}{\phi}

\newcommand{\la}{\lambda}

\newcounter{cnt1}
\newcounter{cnt2}
\newcounter{cnt3}
\newcommand{\blr}{\begin{list}{$($\roman{cnt1}$)$}
        {\usecounter{cnt1} \setlength{\topsep}{0pt}
                \setlength{\itemsep}{0pt}}}
\newcommand{\bla}{\begin{list}{$($\alph{cnt2}$)$}
        {\usecounter{cnt2} \setlength{\topsep}{0pt}
                \setlength{\itemsep}{0pt}}}
\newcommand{\bln}{\begin{list}{$($\arabic{cnt3}$)$}
        {\usecounter{cnt3} \setlength{\topsep}{0pt}
                \setlength{\itemsep}{0pt}}}
\newcommand{\el}{\end{list}}
\newtheorem{thm}{Theorem}

\newtheorem{Def}[thm]{Definition}

\newtheorem{rem}[thm]{Remark}
\newcommand{\Rem}{\begin{rem} \rm}
\newcommand{\bdfn}{\begin{Def} \rm}
\newcommand{\edfn}{\end{Def}}

\title[Subdifferential of the $\mathcal{B(H,K)}$ norm, and approximate orthogonality]{Subdifferential of the $\mathcal{B(H,K)}$ norm, and approximate orthogonality}
\author[ P. Grover, \ K. K. Gupta, \ S. Seal ]
{Priyanka Grover$^{1}$, Krishna Kumar Gupta$^{2}$, Susmita Seal$^{3}$ }
\address{{$^{1}$}   Priyanka Grover,
	Department of Mathematics, 
	Shiv Nadar Institution of Eminence Delhi NCR, NH-91, Tehsil Dadri, Uttar Pradesh 201314, India 
}
\email{priyanka.grover@snu.edu.in}

\address{{$^{2}$}   Krishna Kumar Gupta,
	Department of Mathematics, 
	Shiv Nadar Institution of Eminence Delhi NCR, NH-91, Tehsil Dadri, Uttar Pradesh 201314, India 
}
\email{kg952@snu.edu.in, shrikrishna6996@gmail.com}

\address {{$^{3}$} Susmita Seal, 
	School of Mathematical Sciences, National Institute of Science Educational and Research Bhubaneswar, An OCC of Homi Bhabha National Institute, P.O. - Jatni, District - Khurda, Odisha - 752050, India}
\email{susmitaseal1996@gmail.com, susmitasealmath@niser.ac.in}

\subjclass{ 15A60, 47L10, 46G05, 46B20, 47A30}
\keywords{ Normed spaces, Linear
	operators, Compact operators, Subdifferential, Norm derivative, Maximal numerical range, Birkhoff-James orthogonality, Approximate orthogonality}
\date{}
\begin{document}

	\begin{abstract}
		We present an expression for the right hand derivative of the $\mathcal{B(H,K)}$ norm generalizing the result for $\K=\H$  in [D. J. Ke$\check{\mathrm{c}}$ki$\grave{\mathrm{c}}$, {\it Gateaux derivative of $B(H)$ norm}, Proc. Amer. Math. Soc. {\bf 133} (2005): 2061--2067]. Using this, we obtain the subdifferential of the $\mathcal{B(H, K)}$ norm. For tuples of operators $\A,\X\in$ $\mathcal{B(H, H}^d)$, we give a characterization for $\boldsymbol 0$ to be a best approximation to the subspace $\mathbb C^d \X$, generalizing a similar result for $\mathbb C^d \I$ in [P. Grover, S. Singla, {\it A distance formula for tuples of operators}, Linear Algebra Appl. {\bf 650} (2022): 267--285]. We define the concept of $\epsilon$-Birkhoff orthogonality to a subspace in a general normed space and derive a characterization in terms of the subdifferential set. Using this, we deduce interesting results for $A\in \mathcal{B(H,K)}$ to be $\epsilon$-Birkhoff orthogonal to a subspace of $\mathcal{B(H,K)}$, when $A$ is compact.
		
	\end{abstract}
	\maketitle

%\maketitle

\section{Introduction}

Let $(\mathcal{X}, \|\cdot\|)$ be a normed space.
 The right hand derivative of $\|\cdot\|$ at $a\in \mathcal X$ is given by \begin{equation}\label{rd}\|a\|^{'}_{+}(x)=\lim\limits_{t\rightarrow 0^+} \frac{\|a+tx\| -\|a\|}{t}, \quad x \in \mathcal{X}.\end{equation}
 Let $\mathcal{H}$ be a complex Hilbert space.  Let $\mathcal{B(H)}$ denote the space of bounded linear operators from $\mathcal{H}$ to $\mathcal{H}.$ The expression for the right hand derivative of the $\mathcal{B(H)}$ norm, $\|\cdot\|^{'}_{+}$, was derived in \cite{K}. It was then used to derive characterization of \emph{Birkhoff-James orthogonality}  in $\mathcal B( H)$. Recently, properties of norm derivatives have been used as sharp tools in \cite{enderami, ks, sain, sain2}, to name a few. Finding exact expressions of $\|\cdot\|^{'}_+$ for various norms is not a trivial task, and it has been a subject of interest. For matrix norms, this has been done in \cite{G, GS2, W}. In $C^*$-algebras, an expression is given in \cite{SsinglaC*}. In any normed space $\mathcal X$, the right hand derivative serves as the support functional for the \emph{subdifferential set}, that is, for $a\in \mathcal X$, the subdifferential set of $\|\cdot\|$ at $a$ is given by $$\partial \|a\|=\{x^* \in \mathcal{X}^*:\  \mathrm{Re} \ x^*(x) \leqslant \|a\|^{'}_{+}(x) \  \text{ for every }  x\in \mathcal{X}\}.
$$ It is also same as $\partial \|a\|=\{f\in \mathcal X^*: \|f\|=1, f(a)=\|a\|\},$ for $a\neq 0.$ Characterizations of subdifferential sets of matrix norms have been done in \cite{G1, G, GS2, GKpkNorm, W, W1, Z, Z1, Z2}. This concept has been applied to approximation problems in \cite{Singer}, and to Birkhoff-James orthogonality in \cite{GBh, G1, GSubspace,  GS2}. We would like to emphasize that this approach has yielded stronger results in the past, for example compare \cite[Corollary 1.1]{GS2} and \cite[Theorem 2.11]{MP}, where the latter gives a sufficiency result but using subdifferential, necessary part is also obtained in \cite{GS2}. (Also see \cite{GSubspace}.) We refer the readers to the survey \cite{SGsurvey} for more insights. For a recent usage in minimal compact operators of this approach, see \cite{Bot}. In variational analysis, subdifferential set is a key ingredient (see \cite{ding, MH}). 

Let $\mathscr{K}(\mathcal{H})$ denote the space of compact operators from $\H$ to $\H.$ For $T \in \mathcal{B(H)}$, if\\ $\operatorname{dist}(T,\mathscr{K}(\mathcal{H}))<\|T\|$, the subdifferential of the $\mathcal{B(H)}$ norm at $T$ is studied in \cite{SsinglaC*}. Extending these results, in \cite{R},  for a Banach space $\mathcal{X}$, the subdifferential of the norm at $x\in \mathcal X$ has been discussed, under the condition that $\operatorname{dist}(x,J)<1$ for some $M$-ideal $J \subset \mathcal{X}$. Let $\mathcal {B(H,K)}$ denote the space of all bounded operators from Hilbert space $\H$ to another Hilbert space $\K.$ We give a characterization of the subdifferential of the operator norm on $\mathcal {B(H,K)}$ at $A\in \mathcal B(\mathcal H, \mathcal K)$ without any condition on $A$.
 To achieve this, we first need to establish an explicit expression for $\|A\|^{'}_{+}$, the right hand  derivative of $\|\cdot\|$ at $A\in\mathcal{B(H,K)}$, which we do in Theorem \ref{D2}. 

 For $d\in\N,$ let $\mathcal{H}^d$ be the direct sum of $d $ copies of the Hilbert space $\mathcal{H}$, equipped with the $\ell_2$-norm. For $\mathcal{K}=\mathcal{H}^d,$ let $\mathcal{B(H, H}^d)$ be the space of bounded linear operators from $\mathcal{H}$ to $\mathcal{H}^d$. Given elements $A_1,\ldots, A_d\in \mathcal{B(H)}$, we define $\A=(A_1,\hdots, A_d) \in \mathcal{B(H,H}^d)$
 such that, 
 for $\phi\in\mathcal{H},$ $\A\phi= (A_1\phi,\ldots,A_d\phi)$. 
 Let \(\mathbf{I}\) denote the tuple of identity operators $(I, \dots, I)\in\mathcal{B(H,H}^d).$ For $\boldsymbol{\lambda} = (\la_1, \dots, \la_d) \in \mathbb{C}^d,$ we define $\boldsymbol{\la I}$ as the tuple $(\la_1 I, \dots, \la_d I).$ The distance of $\mathbf{A}$ from $\mathbb{C}^d \mathbf{I}$ is given by  
 $ \operatorname{dist}(\mathbf{A}, \mathbb{C}^d \mathbf{I}) = \min\limits_{\boldsymbol{\la} \in \mathbb{C}^d} \|\mathbf{A} - \boldsymbol{\la}\mathbf{I}\|.$ In \cite{GS}, some conditions are considered as to when we have $\mathrm{dist} (\A, \mathbb{C}^d \mathbf{I})=\|\A\|.$ In Theorem \ref{ap1}, leveraging the subdifferential of $\mathcal{B(\mathcal H, \mathcal H}^d)$ norm,  we establish equivalent conditions for $\mathrm{dist} (\A, \mathbb{C}^d \mathbf{X})=\|\A\|.$
 
 For $x,y\in\mathcal{X},$ if $\mathrm{dist}(x,\mathrm{span}\{y\})=\|x\|$, then we say $x$ is \emph{Birkhoff-James orthogonal} to $y$ \cite{B,J}.   Subsequently, concepts of approximate orthogonality  were introduced in \cite{C,D}. In \cite{C}, the notion of \emph{$\epsilon$-Birkhoff orthogonality} was introduced as follows.  
  For $x,y\in\mathcal{X}$ and $\epsilon \in [0,1)$, $x$ is said to be $\epsilon$-Birkhoff orthogonal to $y$ if $\|x+\lambda y\|^2\geqslant \|x\|^2-2\epsilon\|x\| \|\lambda y\| \ \text{for all }  \lambda \in \C.$ We denote this relation by  $x\perp_B^\epsilon y.$ A characterization in the space of bounded linear operators defined on real normed spaces is given in \cite{PSM}. For finite dimensional case of real Hilbert spaces, it was obtained in \cite{C1}. Much recently,  characterizations were given for general normed spaces in \cite{ACCF} in terms of norm derivatives. Some characterizations for $\K=\H$, that is, $\mathcal {B(H)}$, are pointed out as special cases in \cite{C1} and \cite{PSM}. Using the right-hand derivative of the $\mathcal{B(H,K)}$ norm, Theorem \ref{epsilon_in_Bhk} gives a characterization of $\epsilon$-Birkhoff orthogonality in $\mathcal{B(H,K)}$. In \cite{PSM}, $\epsilon$-Birkhoff orthogonality for $\mathscr{K} (\mathcal{X},\mathcal{Y})$, the space of compact operators from a real normed space $\mathcal X$ to another real normed space $\mathcal Y$ has been discussed. In Theorem \ref{epsilonBirkChara 1}, for complex Hilbert spaces, we characterize the $\epsilon$-Birkhoff orthogonality of a compact operator $A$ with respect to an arbitrary bounded operator $B \in \mathcal{B(H,K)}$.

  We extend the notion of $\epsilon$-Birkhoff orthogonality to a subspace of a normed space $\mathcal{X}.$ Let $\mathcal{W}$ be a subspace of $\mathcal{X}.$ For $x\in\mathcal{X},$ we say $x$ is $\epsilon$-Birkhoff orthogonal to $\mathcal{W}$ if 
  \begin{equation}\label{sub_ortho_def}\|x+w\|^2\geqslant \|x\|^2-2\epsilon \|x\|\|w\|\ \text{ for all }w\in\mathcal{W}.
  	\end{equation} We denote it as $x\perp_B^\epsilon \mathcal{W}.$ In Theorem \ref{subspace_epsilon1}, we give a characterization of $\epsilon$-Birkhoff orthogonality to a subspace, analogous to \cite{C1,PW4} in terms of the subdifferential set. We also present a characterization of $\epsilon$-Birkhoff orthogonality of a compact operator to a subspace in $\mathcal{B(H,K)}.$

In Section 2, we obtain the explicit expression for $\|A\|^{'}_{+}$ and subsequently derive the subdifferential set of $\|\cdot\|$ at $A\in\mathcal{B(H,K)}.$ 
In Section 3, we provide a characterization for $\boldsymbol{0}$ to be a best approximation to the subspace $\mathbb C^d \X$ in $\mathcal B(\mathcal H, \mathcal H^d).$  
 In Section 4, as an application of the subdifferential set, we provide a characterization of \(\epsilon\)-Birkhoff orthogonality to a subspace of a general normed space. We also present some results that characterize $\epsilon$-Birkhoff orthogonality in $\mathcal{B(H,K)}.$

\section{Subdifferential and right hand derivative of $\mathcal{B(H, H}^d)$ norm}

 Before presenting our main result, we introduce the following notation. Let $E_{A^*A}$ be the spectral measure of the operator $A^*A.$
For $\delta >0$ and $A\in \mathcal{B(H,K)}$, we define $H_{\delta}(A) := E_{A^*A}\left[\|{A}\|^2-\delta, \|{A}\|^2\right].$
 If $A\in \mathcal{B(H)}$ is self adjoint, we set $\widetilde{H_{\delta}(A)} := E_{A}[\|A\|-\delta, \|A\|]$, where $E_{A}$ is the spectral measure of the operator $A$.

For a nonzero operator $A\in \mathcal{B(H,K)}$, we define the set $\Lambda (A) :=\{\Gamma=(\phi_{n})_{n} : \phi_n\in\mathcal{H},\|\phi_{n}\|=1\text{ for all }n\in\mathbb{N}\ \mathrm{with} \ \|A\phi_{n}\|\rightarrow \|A\|\}.$
Additionally, let $\mathcal{G}$ denote the set of \emph{Banach limits} on the space $\ell^{\infty},$ the space of all complex valued bounded sequences. For each $\Gamma \in \Lambda (A)$ and $g\in\mathcal{G},$ we define the function  $f_{g,\Gamma} : \mathcal{B(H,K)}\rightarrow \mathbb{C}$ as
 $$f_{g,\Gamma} (X)=g\left( \left\{\frac{ \langle X\phi_{n},\ A\phi_{n}\rangle}{\|A\|}\right\}\right) \quad \text{ for all } X\in \mathcal{B(H, K)}.$$

We now derive an expression for the right hand derivative for a operator in $\mathcal{B}(\H,\K).$

\begin{lemma}\label{K2}
\cite{K}
Let $X$, $Y$ and $Z$ be self adjoint operators in $\mathcal{B(H)}$ such that $X$ and $Z$ are positive. 
Then for all $\delta>0,$
$$\lim\limits_{t\rightarrow 0^+} \frac{\|X+tY+t^2Z\| - \|X\|}{t} \leqslant \sup\limits_{\substack{\phi \in \widetilde{H_{\delta}(X)} \\ \|\phi\|=1}} \langle Y\phi,\phi\rangle.$$ 
\end{lemma}

\begin{theorem}\label{D2}
Let $A, X\in \mathcal{B(H,K)}$ with $A\neq 0$. Then 
$$\lim\limits_{t\rightarrow 0^+} \frac{\|A+tX\| -\|A\|}{t}= \frac{1}{\|A\|}\inf\limits_{\delta >0}\sup\limits_{\substack{\phi\in H_{\delta}(A)\\ \|\phi\|=1}}  \ \mathrm{Re} \langle X\phi,\ A\phi\rangle.$$
\end{theorem}
\begin{proof}
We begin by noting that
\begin{align}
\frac{\|A+tX\|-\|A\|}{t} & =  \frac{\|A+tX\|^2-\|A\|^2}{t (\|A+tX\|+\|A\|)}\nonumber\\
&= \frac{\| (A+tX)^*(A+tX)\|-\|A\|^2}{t (\|A+tX\|+\|A\|)} \nonumber\\
&= \frac{\|A^* A +t(X^*A+A^*X)+ t^2 X^*X\|-\|A\|^2}{t (\|A+tX\|+\|A\|)}.\label{eq1}
\end{align}
Applying Lemma $\ref{K2}$, for each $\delta >0$, we obtain
\begin{equation}\notag
\begin{aligned}
\lim\limits_{t\rightarrow 0^+} \frac{\|A+tX\| -\|A\|}{t}  \leqslant 
\frac{1}{2\|A\|} \sup\limits_{\substack{\phi\in H_{\delta}(A)\\ \|\phi\|=1}} \ \langle (X^*A+A^*X)\phi, \phi\rangle
 =\frac{1}{\|A\|} \sup\limits_{\substack{\phi\in H_{\delta}(A)\\ \|\phi\|=1}} \  \mathrm{Re} \langle X\phi, A\phi\rangle.
\end{aligned}
\end{equation}
Taking the infimum over  $\delta >0$ yields
\begin{equation}\label{eq2}
\lim\limits_{t\rightarrow 0^+} \frac{\|A+tX\| -\|A\|}{t}\leqslant \frac{1}{\|A\|}\inf\limits_{\delta >0}\sup\limits_{\substack{\phi\in H_{\delta}(A) \\ \|\phi\|=1}} \ \mathrm{Re} \langle X\phi, \ A\phi\rangle. 
\end{equation}
For the reverse inequality, let $\delta >0$. Choose $\phi_{\delta}\in H_{\delta}(A)$ with $\|\phi_{\delta}\|=1$ such that 
$$ \mathrm{Re}  \langle X\phi_{\delta},\ A\phi_{\delta}\rangle \geqslant \sup\limits_{\substack{\phi\in H_{\delta}(A) \\ \|\phi\|=1}} \  \mathrm{Re} \langle X\phi,\ A\phi\rangle  -\delta.$$
Also $\lim\limits_{\delta\rightarrow0^+} \langle \Big( A^* A\Big)\phi_{\delta}, \phi_{\delta}\rangle = \|A\|^2$. Hence, from ($\ref{eq1}$),

\begin{equation}\notag
	\begin{aligned}
		\frac{\|A+tX\| - \|A\|}{t} 
		&\geqslant \frac{1}{\|A+tX\|+\|A\|} \Bigg( \frac{1}{t} \Big( \langle  A^* A  \phi_{\delta}, \phi_{\delta} \rangle - \|A\|^2 \Big) 
		+ 2\ \Re \langle X \phi_{\delta}, A \phi_{\delta} \rangle \\
		&\hspace{6cm} + t \langle   X^* X  \phi_{\delta}, \phi_{\delta} \rangle \Bigg) \\
		&\geqslant \frac{1}{\|A+tX\|+\|A\|} \Bigg( \frac{1}{t} \Big( \langle  A^* A  \phi_{\delta}, \phi_{\delta} \rangle - \|A\|^2 \Big) \\
		&\hspace{2cm} + 2 \sup_{\substack{\phi\in H_{\delta}(A) \\ \|\phi\|=1}}  \mathrm{Re} \langle X \phi, A \phi \rangle 
		- 2\delta + t \langle  X^* X \phi_{\delta}, \phi_{\delta} \rangle \Bigg).
	\end{aligned}
\end{equation}

By taking $\liminf\limits_{\delta\rightarrow 0+}$, we get
\begin{equation}
\begin{aligned}
\frac{\|A+tX\| -\|A\|}{t} \geqslant \frac{1}{\|A+tX\|+\|A\|} \Bigg (2 \inf\limits_{\delta >0} \sup\limits_{\substack{\phi\in H_{\delta}(A)\\ \|\phi\|=1}} \ \mathrm{Re} \langle X\phi,\ A\phi\rangle 
 + t \liminf\limits_{\delta\rightarrow 0+} \langle X^*X\phi_{\delta}, \phi_{\delta}\rangle \Bigg).
\end{aligned}
\end{equation}

Therefore, 
\begin{equation}\label{eq3}
\begin{split}
\lim\limits_{t\rightarrow 0^+} \frac{\|A+tX\| -\|A\|}{t} \geqslant \frac{1}{\|A\|}\inf\limits_{\delta >0}\sup\limits_{\substack{\phi\in H_{\delta}(A) \\ \|\phi\|=1}} \ \mathrm{Re} \langle X\phi,\ A\phi\rangle. 
\end{split}
\end{equation}
By combining ($\ref{eq2}$) and ($\ref{eq3}$), the proof is completed.
\end{proof}

\begin{corollary}\label{finite_righhand_1}
 Let $A,X \in\mathcal{B(H,K)}$ be such that $A$ is compact operator. Then $$\lim\limits_{t\rightarrow 0^+} \frac{\|A+tX\| -\|A\|}{t}=\frac{1}{\|A\|}\max\limits_{\substack{\phi\in \mathcal{H},\|\phi\|=1,\\ A^*A \phi=\|A\|^2\phi}}\ \mathrm{Re} \langle X\phi,\ A\phi\rangle.$$
\end{corollary}

\begin{proof}
	Since $A$ is compact operator, we get $\|A\|^2$ is an eigenvalue of $A^*A.$ So, $
\bigcap\limits_{\delta>0}{H_\delta}(A)=\{\phi\in \mathcal{H}:A^*A\phi=\|A\|^2\phi\}$ is a finite dimensional subspace. 
Furthermore,  as $\delta\rightarrow 0^+,$  the sets $H_{\delta}(A)$ form a nested family. Let $\rho= \max\limits_{\substack{\phi\in {\bigcap\limits_{\delta>0}{H_\delta}(A)}\\ \|\phi\|=1}} \  \mathrm{Re} \langle X\phi,\ A\phi\rangle.$ Now, we claim that
$\inf\limits_{\delta >0}\sup\limits_{\substack{\phi\in H_{\delta}(A) \\ \|\phi\|=1}} \ \mathrm{Re} \langle X\phi,\ A\phi\rangle=\rho.$ It is easy to see that 
$$\inf\limits_{\delta >0}\sup\limits_{\substack{\phi\in H_{\delta}(A) \\ \|\phi\|=1}} \ \mathrm{Re} \langle X\phi,\ A\phi\rangle
\geqslant \rho.$$  
Suppose, $$\inf\limits_{\delta >0}\sup\limits_{\substack{\phi\in H_{\delta}(A) \\ \|\phi\|=1}} \ \mathrm{Re} \langle X\phi,\ A\phi\rangle
>\rho.$$ So, we can choose an $\varepsilon_0>0$ such that $\inf\limits_{\delta >0}\sup\limits_{\substack{\phi\in H_{\delta}(A) \\ \|\phi\|=1}} \ \mathrm{Re} \langle X\phi,\ A\phi\rangle
>\rho+\varepsilon_0.$
 Then, $$\sup\limits_{\substack{\phi\in H_{\delta}(A) \\ \|\phi\|=1}} \ \mathrm{Re} \langle X\phi,\ A\phi\rangle
>\rho+\varepsilon_0 \text{ for all }\delta>0.$$ This implies for each $n\in\N,$
there exists $\phi_n\in H_{\frac{1}{n}}(A)$ such that $\|\phi_n\|=1,$ and $\Re\langle X\phi_n,A\phi_n\rangle>\rho+\varepsilon_0.$ Since the closed unit ball in $\H$ is weakly compact, there exits a subsequence $\{\phi_{n_k}\}$ of $\{\phi_n\}$ such that $\phi_{n_k}$ converges weakly to $\phi_0\in\H.$ Also, the compactness of $A$ implies that there exits a subsequence $\{\phi_{n_{k_l}}\}$ of $\{\phi_{n_k}\}$ such that $A\phi_{n_{k_l}}$ converges to $\varphi_0\in\K.$ By the uniqueness of weak limit, we get $A\phi_0=\varphi_0.$ Since $\|A\ph_n\|$ converges to $\|A\|,$ we get $\|A\ph_0\|=\|A\|.$ This implies that $\|\phi_0\|=1$ and $\phi_0\in {\bigcap\limits_{\delta>0}{H_\delta}(A)}.$ Thus, we get
$$\Re\langle X\phi_0,A\phi_0\rangle=\lim_{l\to\infty}\Re\langle X \phi_{n_{k_l}},A\phi_{n_{k_l}}\rangle \geq \rho+\varepsilon_0> \Re\langle X\phi_0,A\phi_0\rangle,$$ which is a contradiction. Therefore, \begin{align*}
	\inf\limits_{\delta >0}\sup\limits_{\substack{\phi\in H_{\delta}(A) \\ \|\phi\|=1}} \ \mathrm{Re} \langle X\phi,\ A\phi\rangle
 =\max\limits_{\substack{\phi\in \mathcal{H},\|\phi\|=1,\\ A^*A \phi=\|A\|^2\phi}} \ \mathrm{Re} \langle X\phi,\ A\phi\rangle.
\end{align*}  
Thus, the desired result follows.  
\end{proof}

We are now prepared to prove our main result. We denote $\overline{\mathcal{C}}^{w^*}$ as the weak$^*$-closure of a set $\mathcal{C}.$ The notation $\mathrm{conv}\{\mathcal C\}$ stands for the convex hull of the set $\mathcal C$.

\begin{theorem}\label{main1}
Let $A\in \mathcal{B(H,K)}$ with $A\neq 0$. Then $\partial\|A\|=\overline{\mathrm{conv}}^{w^*}\{f_{g,\Gamma} : g\in\mathcal{G},\Gamma\in \Lambda (A)\}.$
\end{theorem}
\begin{proof}
Define $\mathcal{M}:=\overline{\mathrm{conv}}^{w^*}\{f_{g,\Gamma} : g\in\mathcal{G},\Gamma\in \Lambda (A)\}.$
We first observe that for $g\in\mathcal{G}$ and $\Gamma\in \Lambda (A)$,
\begin{equation}\notag
\begin{split}
f_{g,\Gamma} (A)=g\left( \left\{\frac{ \langle A\phi_{n},\ A\phi_{n}\rangle}{\|A\|}\right\}\right)
=g\left( \left\{\frac{\|A \phi_{n}\|^2}{\|A\|}\right\}\right)
=\frac{\|A\|^2}{\|A\|}=\|A\|.
\end{split}
\end{equation}
For $g\in\mathcal{G}$ and $\Gamma \in \Lambda (A)$, $\|f_{g,\Gamma}\|= 1$. This implies that $\mathcal{M} \subset \partial\|A\|$. Suppose that  $\mathcal{M} \subsetneq \partial\|A\|$. Then there exists $f_0\in \partial\|A\|$ such that $f_0\notin \mathcal{M}$. By the Hahn-Banach separation theorem, there exists $X\in \mathcal{B(H, K)}$ and $\alpha\in \mathbb{R}$ such that 
$$\sup\limits_{f\in \mathcal{M}} \ \mathrm{Re} \ f(X) < \alpha < \mathrm{Re} \ f_0(X).$$
So, for every $g\in\mathcal{G}$ and $(\phi_n)_n \in \Lambda (A),$ $$\mathrm{Re} \ g\left( \left\{\frac{ \langle X\phi_{n},\ A\phi_{n}\rangle}{\|A\|}\right\}\right) < \alpha < \mathrm{Re} \ f_0(X).$$
This gives \begin{equation}g\left( \left\{\frac{\mathrm{Re} \langle X\phi_{n},\ A\phi_{n}\rangle}{\|A\|}\right\}\right) < \alpha < \mathrm{Re} \ f_0(X) \quad \text{for all } g\in\mathcal{G}\text{ and }(\phi_n)_n \in \Lambda (A).\label{contradiction}\end{equation}
We now claim that \ $\inf\limits_{\delta >0}\sup\limits_{\substack{\phi\in H_{\delta}(A) \\ \|\phi\|=1}} \ \mathrm{Re} \  \frac{ \langle X\phi,\ A\phi\rangle}{\|A\|} \leqslant \alpha$. 
If this were not the case, then for each $n\in \mathbb{N}$, we would have  $$\sup\limits_{\substack{\phi\in H_{\frac{1}{n}}(A)\\ \|\phi\|=1}} \ \mathrm{Re} \  \frac{ \langle X\phi,\ A\phi\rangle}{\|A\|} > \alpha.$$
Thus, there exists $\phi_n \in  H_{\frac{1}{n}}(A)$ with $\|\phi_n\|=1$  such that $\mathrm{Re} \  \frac{\langle X\phi_n,\ A\phi_n\rangle}{\|A\|} > \alpha$.
 Hence, $(\phi_n)_n \in \Lambda(A)$ and it follows that, for all $g\in\mathcal{G},$ $g\left( \left\{\frac{\mathrm{Re} \langle X\phi_{n},\ A\phi_{n}\rangle}{\|A\|}\right\}\right) \geqslant \alpha$, which contradicts our earlier inequality \eqref{contradiction}. 
 
 Therefore,
$$\inf\limits_{\delta >0}\sup\limits_{\substack{\phi\in H_{\delta}(A) \\ \|\phi\|=1}} \ \mathrm{Re} \  \frac{ \langle X\phi,\ A\phi\rangle}{\|A\|} \leqslant \alpha < \mathrm{Re} \ f_0(X).$$
Since left hand side represents the directional derivative of $\|A\|$ in the direction $X$, this  contradicts  Theorem $\ref{D2}$. Consequently, our assumption that $\mathcal{M} \subsetneq \partial\|A\|$ must be false, proving the theorem.

\end{proof}

\begin{corollary}\label{subdiff_comp_1}
	Let $A\in\mathcal{B(H,K)}$ be a non-zero compact operator. Then
	\begin{equation}\label{subdiff_finite_eq2}
		\partial \|A\|=\mathrm{conv}\left\{f_{\phi}  : \phi\in \mathcal{H},\|\phi\|= 1 \text{ and }  A^*A\phi=\|A\|^2\phi\right\},
	\end{equation}
where $f_\phi(X)=\frac{1}{\|A\|}\langle X\phi,A\phi\rangle$ for all $X\in\mathcal{B(H,K)}.$		
	
\end{corollary}
\begin{proof}
By Corollary \ref{finite_righhand_1}, the proof follows along the same lines as that of Theorem \ref{main1}.
 
\end{proof}

In the following corollary, we note the expression for subdifferential of $\mathcal{B(H, H}^d),$ which will be useful in later discussion.

\begin{corollary}\label{BHdsubdiff}
	Let $\A=(A_1,\ldots,A_d)\in\mathcal{B(H, H}^d)$ with $\A\neq 0$. Then $$\partial\|\A\|=\overline{\mathrm{conv}}^{w^*}\{f_{g,\Gamma} :g\in\mathcal{G}, \Gamma\in \Lambda (\A)\},$$ 
	where
	$f_{g,\Gamma} (\X)=g \left(\left\{\frac{\sum\limits_{i=1}^{d} \langle X_i\phi_{n},\ A_i\phi_{n}\rangle}{\|\A\|}\right\}\right) \text{ for all } \X=(X_1,\ldots,X_d)\in \mathcal{B(H, H}^d).$
	
\end{corollary}

\section{Approximation in $\mathcal {B(H, H}^d)$ }

As a direct consequence of Theorem \ref{D2}, we get the following proposition for $\mathcal{B(H, H}^d).$ This will be helpful for our subsequent discussions.
\begin{proposition}\label{KeB1}

Let $\A, \X \in \mathcal{B(H,H}^d).$ 
 Then 
 $\|\A+\lambda \X\|\geqslant \|\A\|$ for all $\lambda \in \mathbb{C}$ if and only if for each $\delta >0$ and $\theta \in [0,2\pi)$,
$\sup\limits_{\phi\in H_{\delta}(\A), \|\phi\|=1} \sum\limits_{j=1}^{d} \mathrm{Re} (e^{i\theta} \langle  X_j\phi,\ A_j\phi\rangle)\geqslant 0$.
\end{proposition}

\begin{proof}
From \cite[Proposition 1.5]{K}, we have 
$\|\A+\lambda \X\|\geqslant \|\A\|$ for all $\lambda \in \mathbb{C}$ if and only if $$\lim\limits_{t\rightarrow 0+} \frac{\|\A+te^{i\theta} \X\| -\|\A\|}{t} \geqslant 0 \quad \text{ for all } \theta \in [0,2\pi).$$
Moreover, $$\lim\limits_{t\rightarrow 0+} \frac{\|\A+te^{i\theta} \X\| -\|\A\|}{t} = \frac{1}{\|\A\|} \inf\limits_{\delta >0}\sup\limits_{\phi\in H_{\delta}(\A), \|\phi\|=1} \sum\limits_{j=1}^{d} \mathrm{Re} (e^{i\theta} \langle X_j\phi,\ A_j\phi\rangle).$$
Hence, the result follows.
\end{proof}

For $\boldsymbol{\lambda}=(\lambda_1,\ldots,\lambda_d) \in \mathbb C^d$ and $\X=(X_1,\ldots,X_d)\in \mathcal{B(H,H}^d)$, $\boldsymbol{\lambda}\X$ denotes the tuple $(\lambda_1 X_1,\ldots,\lambda_d X_d)$. Let $S(\boldsymbol{\lambda}) = \boldsymbol{\lambda} \X$ for $\boldsymbol{\lambda}\in \mathbb C^d.$  
For $\A, \X\in \mathcal{B(H, H}^d)$, the joint maximal numerical range of $\A$ with respect to $\X$ is defined as 
$$
\begin{aligned}
	W_0(\A,\X) := \Big\{(c_1,\hdots,c_d) \in \mathbb{C}^d \ \Big| \   c_i=\lim\limits_{n\rightarrow \infty} \langle X_i\phi_n, A_i\phi_n\rangle \ \text{for all} \ i=1,\hdots,d, 
	\text{where } \phi_n\in\mathcal{H},\\ \|\phi_n\|=1 \ \text{for all } n\in\mathbb{N} \ \text{ and } \lim\limits_{n\to\infty}\|\A\phi_n\|=\|\A\| \Big\}.
\end{aligned}
$$
We denote $W_0(\A,\I)$ by $W_0(\A)$.
Let $\boldsymbol{\lambda}^0=(\lambda^0_1,\ldots, \lambda^0_d)\in \mathbb{C}^d$ be the unique element such that $\operatorname{dist}(\A, \mathbb{C}^d \I)=\|\A-\boldsymbol{\lambda}^0\I\|$. Define $\A^0=\A-\boldsymbol{\lambda}^0\I$ and for each $1\leqslant j\leqslant d$, let $A_j^0=A_j-\lambda_j^0I$. It was shown in \cite[Prop. 8]{GS} that if $W_0(\A)$ is convex, then $\boldsymbol 0$ is the best approximation to the subspace $\mathbb C^d \I$ if and only if $\boldsymbol 0\in W_0(\A)$. This follows as a special case of our next theorem.
\begin{theorem}\label{ap1}
Let $\A,\X\in \mathcal{B(H,H}^d)$. 
Then the following are equivalent.
\begin{enumerate}
\item $\|\A+\boldsymbol{\lambda} \X\|\geqslant \|\A\|$ for all $\boldsymbol{\lambda} \in \mathbb{C}^d$.
\item $(0,\hdots,0)\in\operatorname{conv} W_0(\A,\X).$

\item $(0, \dots, 0) \in S^*(\partial \|\A\|).$ 

\end{enumerate}
\end{theorem}

\begin{proof}
(i) $\Rightarrow$ (ii). 
From (i), we have for each  $\lambda \in \mathbb{C}$, and for each $(\lambda_1,\hdots,\lambda_d)\in \mathbb{C}^d$,
\begin{equation}\notag
\begin{split} 
\Big\|\A + \lambda (
\lambda_1 X_1,\ldots,\lambda_d X_d)\Big\| \geqslant \Big\|\A \Big\|. 
\end{split}
\end{equation}
Then from Proposition ($\ref{KeB1}$), it follows that, for each $\delta >0$ and $\theta \in [0,2\pi)$,
\begin{equation*}
\sup\limits_{\phi\in H_{\delta}(\A), \|\phi\|=1} \sum\limits_{j=1}^{d} \mathrm{Re}  (e^{i\theta} \langle \lambda_j X_j\phi,\ A_j\phi\rangle)\geqslant 0 \quad \text{ for all } (\lambda_1,\hdots,\lambda_d)\in \mathbb{C}^d.
\end{equation*}
So for each $(\lambda_1,\hdots,\lambda_d)\in \mathbb{C}^d$,
\begin{equation}\label{equn}
 \sup\limits_{\phi\in H_{\delta}(\A), \|\phi\|=1} \sum\limits_{j=1}^{d} \mathrm{Re} (\lambda_j \langle  X_j\phi,\ A_j\phi\rangle)\geqslant 0 .
\end{equation}

We claim that $(0,\hdots,0)\in\operatorname{conv} ( W_0(\A,\X)).$
If not, then there exists $(\eta_1,\hdots,\eta_d)\in \mathbb{C}^d$ and $\alpha\in \mathbb{R}$ such that 
$$\mathrm{Re} \left(\sum\limits_{j=1}^{d} \eta_j c_j\right) < \alpha < 0 \quad \text{ for all } (c_1,\hdots,c_d)\in\operatorname{conv} (W_0(\A,\X)).$$
From ($\ref{equn}$), for each $n\in \mathbb{N}$, we choose $\phi_n \in H_{\frac{1}{n}}(\A)$ with $\|\phi_n\|=1$ such that
$$\mathrm{Re}  \left(\sum\limits_{j=1}^{d} \eta_j \langle  X_j\phi_n,\ A_j\phi_n\rangle\right)\geqslant -\frac{1}{n}.$$
Passing to a subsequence, if necessary,
let $c_j= \lim\limits_{n\rightarrow \infty} \langle  X_j\phi_n,\ A_j\phi_n\rangle$ for $j=1,\hdots,d$. Then $(c_1,\hdots,c_d)\in  W_0(\A,\X)$ and $\mathrm{Re} \sum\limits_{j=1}^{d} (\eta_j c_j) \geqslant 0$, a contradiction. Hence, our claim is true.

(ii) $\Rightarrow$ (iii).
Let $(0,\hdots,0) = \sum\limits_{i=1}^{k} \alpha_{i} (c_{i1},\hdots, c_{id}),$ where $\sum\limits_{i=1}^{k} \alpha_{i} =1$, $\alpha_{i}\geqslant 0$ \text{ for all }$1\leqslant i \leqslant k$.
Then for $1\leqslant i \leqslant k$, there exists $(\phi_{i,m})_m\in \Lambda(\A)$ such that
$c_{ij}=\lim\limits_{m\rightarrow \infty} \langle X_j\phi_{i,m}, A_j\phi_{i,m}\rangle$ for all $  1\leqslant j \leqslant d$.
Let $g\in\mathcal{G},$ $\Gamma_{i}= (\phi_{i,m})_m$. For $\mathbf{Y}=(Y_1,\ldots,Y_d)\in \mathcal{B(H, H}^d)$, let $f_{g, \Gamma_{i}} (\mathbf{Y})=g\left(\left\{ \frac{\sum\limits_{j=1}^{d} \langle Y_j\phi_{i,m},\ A_j\phi_{i,m}\rangle}{\|\A\|}\right\}\right)$.
Define  $f = \sum\limits_{i=1}^{k} \alpha_{i} f_{g, \Gamma_{i}}.$
Then, by Corollary \ref{BHdsubdiff}, $f \in \partial \|\A\|$.
 Also for each $(\lambda_1,\hdots, \lambda_d)\in \mathbb{C}^d$,
\begin{eqnarray*}\notag
\begin{split}
 f (
\lambda_1 X_1,
\ldots,
\lambda_d X_d
)&=  \sum\limits_{i=1}^{k} \alpha_{i}\ g\left(\left\{ \frac{\sum\limits_{j=1}^{d} \langle \lambda_j X_j\phi_{i,m},\ A_j\phi_{i,m}\rangle}{\|\A\|}\right\}\right)\hspace{.2 cm}\\
&=\frac{1}{\|\A\|} \sum\limits_{i=1}^{k} \alpha_{i} (\lambda_1 c_{i1} +\ldots + \lambda_d c_{id})\\
&=0. \hspace{5.6 cm}
\end{split}
\end{eqnarray*}
Therefore, $(0,\hdots,0)=S^*(f) \in S^*(\partial \|\A\|).$

(iii) $\Rightarrow$ (i).
Since $(0,\hdots,0)\in S^*(\partial \|\A\|)$, there exists $f\in \partial \|\A\|$ such that $$f(
\lambda_1 X_1,\ldots,\lambda_d X_d) = 0  \text{ for all }(\lambda_1, \hdots, \lambda_d)\in \mathbb{C}^d.$$
For $\boldsymbol{\lambda}=(\lambda_1, \hdots, \lambda_d)\in\mathbb{C}^d$,
$$\|\A+\boldsymbol{\lambda} \X\|\geqslant f(\A+\boldsymbol{\lambda} \X)= f(
A_1,\ldots,A_d) + f (\lambda_1 X_1,\ldots,\lambda_d X_d
)=\|\A\|.$$
\end{proof}

\section{$\epsilon$-Birkhoff orthogonality in $\mathcal B(\mathcal H, \mathcal K)$}
We present a characterization for $\epsilon$-Birkhoff orthogonality \eqref{sub_ortho_def} to a subspace in Theorem \ref{subspace_epsilon1}. To see that, we recall the following results from subdifferential calculus, which will be useful in the subsequent discussion. 

\begin{proposition}\label{basic_minma_subd} 
	Let $\mathcal{X}$ be a normed space. 
	A continuous convex function $f:\mathcal{X}\to \mathbb{R}$ attains its minima at $a\in\mathcal{X}$ if and only if $0\in \partial f(a).$
\end{proposition}

The following proposition follows from Theorem 2.4.2 and Theorem  2.4.14 of \cite{Zalinescu}. See also \cite{HuL} for the case of $\mathbb{R}^n$ and \cite{G1, Thibault} for more general spaces.

\begin{proposition}\label{Gprop}
	\cite{Zalinescu}
	Let $\mathcal{X}$ and $\mathcal{Y}$ be normed spaces. Consider a bounded linear map  $S: \mathcal{X} \rightarrow \mathcal{Y}$,  continuous affine map $L:\mathcal{X} \rightarrow \mathcal{Y}$ defined by $L(x)=S(x)+y_0$ for some $y_0\in \mathcal{Y}$ and a continuous convex function $g: \mathcal{Y} \rightarrow \mathbb{R}$. Then $ \partial\big( g\, \circ \,L \big)(a) = S^*\partial g (L(a))$ for all $a\in \mathcal{X}$.
\end{proposition}

\begin{proposition}\label{sum}	\cite{Zalinescu}
	Let $f_1, f_2 : \mathcal X \rightarrow \mathbb R$ be two continuous convex functions. Then for $a\in \mathcal{X}$, 
	$$\partial(f_1+f_2)(a)=\partial f_1(a)+\partial f_2(a).$$
\end{proposition}
 
\begin{theorem}\label{subspace_epsilon1}
	Let $\epsilon\in[0,1).$ Let $x\in\mathcal{X},$ and $\mathcal{W}$ be a subspace of $\mathcal{X}.$ Then,
	$x\perp_B^\epsilon\mathcal{W}$ if and only if there exists  $f\in\partial\|x\|$ such that $\|f|_\mathcal{W}\|\leqslant \epsilon.$
	
	\end{theorem}
	\begin{proof}
			Let $ S:\mathcal{W}\to \mathcal{X} $ be inclusion map and $L:\mathcal{W} \to \mathcal{X} $ given by $ L(w)=S(w) + x $ for $w\in\mathcal{W}.$ Additionally, consider the continuous convex function $ h_1:\mathcal{X}\to \mathbb{R} $ given by $ h_1(y)=\|y\|^2 $ and the function $ h_2:\mathcal{W}\to\mathbb{R}^+ $ given by $ h_2(w)=2\epsilon\|x\|\,\|w\|.$ Since $x\perp_B^\epsilon\mathcal{W}$, it follows that $h_1\,\circ L+h_2$ attains its minimum at zero. Then, by  Proposition $\ref{basic_minma_subd}$, $\ref{Gprop}$ and $\ref{sum}$, we obtain $x\perp_B^\epsilon\mathcal{W}$ if and only if
		
			\begin{equation}\label{minimal_sub_equn1} 
			\begin{aligned}
				0 & \in  \partial\big( h_1\, \circ \,L+h_2 \big)(0)\\
				&= S^*\partial h_1(x)+\partial h_2(0)\\
				&= S^*\partial \|x\|^2+ 2 \epsilon \|x\| \{g\in\mathcal{W}^*:\|g\|\leqslant 1\}\\
				&= 2 \|x\|\, S^*\partial (\|x\|)+2 \epsilon \|x\|\{g\in\mathcal{W}^*:\|g\|\leqslant 1\}\\
				&=2\|x\|\,\{f|_\mathcal{W}:f\in\partial\|x\|\}+2 \epsilon \|x\|\{g\in\mathcal{W}^*:\|g\|\leqslant 1\}.
			\end{aligned}
			\end{equation}
			This implies $x\perp_B^\epsilon\mathcal{W}$ if and only if there exist $f\in\partial\|x\|$ and $g\in\mathcal{W}^*,$ $\|g\|\leqslant 1$ such that $f|_\mathcal{W}+g=0.$ This implies that $g=\frac{-f|_\mathcal{W}}{\epsilon}$ for $\epsilon\neq 0,$ and $\|f|_\mathcal{W}\|\leqslant \epsilon.$ This completes the proof.
	\end{proof}
As a consequence of above theorem, we get \cite[Theorem 2.3]{C1} and \cite[Theorem 2.2]{PW4} as follows.
	\begin{corollary}\cite{C1,PW4} 
		Let $\mathcal{X}$ be a real or complex normed space. Let $x,y\in\mathcal{X} $ and $\epsilon\in[0,1).$ Then, $x\perp_B^\epsilon y$ if and only if there exists $f\in\partial \|x\|$ such that $|f(y)|\leqslant \epsilon \|y\|.$
	\end{corollary}
		Let $\mathscr{T}(\H)$ denote the set of trace class operators on $\H.$ For $x\in\H$ and $y\in\K,$ let $y\otimes x\in\mathcal{B(H,K)}$ denote the rank one operator defined as $(y\otimes x)(z)=\langle x,z\rangle y$. The following theorem gives a characterization of $\epsilon$-Birkhoff orthogonality of a compact operator to a subspace in $\mathcal{B(H,K)}.$
	\begin{theorem}\label{epsilonBirsub_BHK}
		Let $A\in\mathcal{B(H,K)}$ be a compact operator. Let $\mathcal{W}$ be a subspace of $\mathcal{B(H,K)}.$ Let $\epsilon\in[0,1)$. Then $A\perp_B^\epsilon \mathcal{W}$ if and only if there exists a positive finite rank operator $P\in\mathscr{T}(\H)$ such that $\tr(P)=1,$  $A^*AP=\|A\|^2P$ and $|\tr(A^*XP)|\leqslant \epsilon \|A\|\|X\| \text{ for all } X\in\mathcal{W}.$  
	\end{theorem}
	\begin{proof}
		Suppose $A\perp_B^\epsilon \mathcal{W},$ then by Theorem \ref{subspace_epsilon1},
		there exists a $f\in\partial\|A\|$ such that $\|f|_\mathcal{W}\|\leqslant \epsilon.$ So, by Corollary \ref{subdiff_comp_1}, there exist numbers $\alpha_1,\ldots,\alpha_{k}$ with $0\leqslant \alpha_j\leqslant 1,$ $\sum\limits_{j=1}^{k}\alpha_j=1,$ and for each $1\leqslant j\leqslant k$, there exists a unit vector $\phi_j\in \H$ satisfying 
		$A^*A\phi_j=\|A\|^2\phi_j$ such that \begin{equation}\label{subBHkeq1}f=\alpha_1f_{\phi_1}+\cdots+\alpha_kf_{\phi_k}.\end{equation}
		Let $P=\alpha_1 \phi_1\otimes\phi_1+\cdots+\alpha_k\phi_k\otimes\phi_k.$ Then, by \eqref{subBHkeq1}, we get $$f(X)=\frac{1}{\|A\|}\sum\limits_{j=1}^k\alpha_j\langle X\phi_j,A\phi_j\rangle=\frac{1}{\|A\|}\tr(A^*XP)\text{ for all }X\in\mathcal{B(H,K)}.$$ Note that $P$ is positive with $\tr(P)=1$ and  $A^*AP=\|A\|^2P.$ Also, $\|f|_\mathcal{W}\|\leqslant \epsilon$ implies that  $\frac{1}{\|A\|}\left|\tr(A^*\frac{X}{\|X\|}P)\right|\leqslant \epsilon$ for all $X\in\mathcal{W}.$ 
		
		Conversely, suppose there exists a positive finite rank operator $P\in\mathscr{T}(\H)$ such that $\tr(P)=1,$  $A^*AP=\|A\|^2P$ and $|\tr(A^*XP)|\leqslant \epsilon \|A\|\|X\| \text{ for all } X\in\mathcal{W}.$ Since $P\in\mathscr{T}(\H)$ is positive finite rank operator, by spectral theorem, there exist orthonormal vectors $\psi_1,\ldots,\psi_m$ such that \begin{equation}\label{new_P1}
			P=\lambda_1\psi_1\otimes\psi_1+\cdots+\lambda_m\psi_m\otimes\psi_m,
		\end{equation} where $0\leqslant\lambda_i\leqslant 1$ and $\sum\limits_{i=1}^m\lambda_i=1.$ So $P\in\mathscr{T}(\H)$ and $\tr(X^*XP)\geqslant 0$ for all $X\in\mathcal{B(H, K)}.$ Since $\mathscr{T}(\H)^*=\mathcal{B(H)},$ it implies that $\|X^*X\|\geqslant \tr(X^*XP)$ for all $X\in\mathcal{B(H,K)}.$ Thus, for $X\in\mathcal{W},$
		\begin{align*}
			\|A+X\|^2=&\|(A+X)^*(A+X)\|\\=&\|A^*A+(X^*A+A^*X)+X^*X\|\\
			\geqslant &\tr(A^*AP)+\tr(X^*AP+A^*XP)+\tr(X^*XP)\\
			\geqslant& \|A\|^2+2\ \Re\ \tr(A^*XP)\\
			\geqslant& \|A\|^2-2\ |\Re\ \tr(A^*XP)|\\
			\geqslant& \|A\|^2-2 \epsilon\|A\|\|X\|.\\
		\end{align*}
		This implies that $A\perp_B^\epsilon \mathcal{W}.$
		
	\end{proof}
	For $\H=\K=\C^n$ and $\epsilon=0,$ we get \cite[Theorem 1]{GSubspace} as the following corollary.
	\begin{corollary}\cite{GSubspace}
		Let $A\in\M_n(\C),$ the space of $n\times n$ complex matrices. Let $\mathcal{W}$ be a subspace of $\M_n(\C).$ Then $A\perp_B
\mathcal{W}$ if and only if there exists a density matrix $P$ such that $A^*AP=\|A\|^2P$ and $AP\in\mathcal{W}^\perp.$	\end{corollary} 
	
Let $\mathbb{D}$ denote the closed unit disc in $\C.$ The next theorem gives a characterization for $\epsilon$-Birkhoff orthogonality when $A\in\mathcal{B(H,K)}$ is compact.
	\begin{theorem}\label{epsilonBirkChara 1}
		Let $A,B\in\mathcal{B(H,K)}$ be such that $A$ is compact operator. Let $\epsilon\in[0,1)$. Then $A\perp_B^\epsilon B$ if and only if there exists a unit vector $\phi\in\H$ satisfying $A^*A\phi=\|A\|^2\phi$ such that 
		$$ |\langle  A^*B\phi,\ \phi \rangle|\leqslant \epsilon \|A\|\|B\|.$$

	\end{theorem}
	\begin{proof}
		By equation~\eqref{minimal_sub_equn1} and Corollary \ref{subdiff_comp_1}, we obtain for $\mathcal{W}=\operatorname{span}\{B\}$ that $A\perp_B^\epsilon B$ if and only if
		
		\begin{equation}\notag  
			\begin{aligned}
				0 & \in
				2\|A\|\mathrm{conv}\{f_{\phi}(B) : \|\phi\|=1,\, A^*A\phi=\|A\|^2\phi\} + 2\epsilon \|A\|\|B\| \mathbb{D}.\\
			\end{aligned} 
		\end{equation} 
	Note that $\{f_{\phi}(B) : \|\phi\|=1,\, A^*A\phi=\|A\|^2\phi\}= \left\{\frac{1}{\|A\|}\langle  A^*B\phi,\ \phi \rangle : \|\phi\|=1, \,A^*A\phi=\|A\|^2\phi \right\}.$ By Toeplitz-Hausdorff theorem, the numerical range of the operator $\frac{1}{\|A\|}A^* B$ restricted to a closed subspace is convex. Thus, we get $A\perp_B^\epsilon B$ if and only if $$0\in  \Big\{\langle  A^*B\phi,\ \phi \rangle : \|\phi\|=1, \,A^*A\phi=\|A\|^2\phi \Big\} + \epsilon \|A\|\|B\| \mathbb{D}.$$
	This gives the required result.
		
	\end{proof}

 Theorem 3.1 of \cite{ACCF} gives a characterization of $\epsilon$-Birkhoff orthogonality in $\mathcal {B(H)}$. We extend it to $\mathcal {B(H, K)}$. For that we need the following lemma, which is similar to Proposition \ref{KeB1}.
\begin{lemma}\label{epsilon_lemma}
	Let $A, B \in \mathcal{B(H,K)}.$ Let $\epsilon\in[0,1)$. Then 
	$A\perp_B^\epsilon B$ if and only if for each $\theta\in[0,2\pi)$
	$$\inf\limits_{\delta>0} \sup\limits_{\phi\in H_{\delta}(A), \|\phi\|=1}  \mathrm{Re} (e^{i\theta} \langle  B\phi,\ A\phi\rangle)\geqslant -\epsilon\|A\|\|B\|.$$
\end{lemma}
\begin{proof}
	From \cite[Theorem 2.2]{ACCF}, $A\perp_B^\epsilon B$ if and only if for all $\theta\in[0,2\pi)$ 
	$$\lim_{t \to 0^+} \frac{\|A + t e^{i\theta} B\|^2 - \|A\|^2}{2t}\geqslant-\epsilon \|A\| \| B\|.$$
	Also, by Theorem \ref{D2}, we have for $\theta\in[0,2\pi)$
	\begin{align*}
		\lim_{t \to 0^+} \frac{\|A + t e^{i\theta} B\|^2 - \|A\|^2}{2t}
		&=\inf\limits_{\delta > 0} \sup\limits_{\substack{\phi \in H_{\delta}(A) \\ \|\phi\| = 1}} 
		 \mathrm{Re} \left( e^{i\theta} \langle B \phi,\ A\phi \rangle \right) .
	\end{align*}
	Hence, the required result follows.
\end{proof}

\begin{theorem}\label{epsilon_in_Bhk}
	Let $A,B\in \mathcal{B(H,K)}.$ Let $\epsilon\in[0,1)$. Then $A\perp_B^\epsilon B$ if and only if for each $\theta\in[0,2\pi)$, there exists a sequence $\phi_n\in\mathcal{H},$ with $\|\phi_n\|=1$ for all $n\in\N,$ such that $$\|A\phi_n\|\to\|A\|\text{ and }\lim_{n\to\infty} \mathrm{Re}(e^{i \theta} \langle A^* B\phi_{n},\ \phi_{n} \rangle)\geqslant -\epsilon \|A\|\|B\|.$$
\end{theorem}
\begin{proof}
	Suppose $A \perp_B^\epsilon B.$ For $n\in\N,$ let $\delta_n=\frac{1}{n}.$ Then, by Lemma~\ref{epsilon_lemma}, for each $n\in\N$ and for each $\theta\in[0,2\pi)$, there exists $\phi_n \in H_{\delta_n}(A)$ with $\|\phi_n\| = 1$ such that
	$$\|A \phi_n\| \to \|A\|\text{ as }n\to\infty \quad \text{and} \quad \lim_{n \to \infty} \mathrm{Re} \left( e^{i \theta} \langle A^* B\phi_n,\ \phi_n \rangle \right) \geqslant -\epsilon \|A\| \|B\|.$$
	(If the sequence does not converge, we can consider a convergent subsequence.)\\
	Conversely, for $\lambda=|\lambda|e^{i \theta} \in\C$
	\begin{align*}
		\|A+\lambda B\|^2&\geqslant \|(A+\lambda B)\phi_n\|^2\\
		&=\left(\|A\phi_n\|^2+2|\lambda|\mathrm{Re}(e^{i \theta}\langle B\phi_n,A\phi_n\rangle)+|\lambda|^2\|B\phi_n\|^2\right)\\
		&\geqslant\|A\phi_n\|^2+2|\lambda|\mathrm{Re}(e^{i \theta}\langle B\phi_n,A\phi_n\rangle).
	\end{align*}
	Taking limit as $n\to\infty,$ we obtain
	$\|A+\lambda B\|^2\geqslant \|A\|^2-2\epsilon \|A\|\|\lambda B\|.$ Since $\lambda\in\C$ was arbitrary, it follows $A\perp_B^\epsilon B.$
\end{proof}

\section*{Acknowledgments}
 The author P. Grover is supported by a SERB Core Research Grant CRG/ $2023/000595$ funded by the Anusandhan National Research Foundation (ANRF), India. A part of this work was done when S. Seal was visiting Shiv Nadar Institution of Eminence, Delhi NCR. She would like to express her deep gratitude for the warm hospitality provided during her visit. Her research was financially supported by the institute Postdoctoral Fellowship of NISER Bhubaneswar.
%\vskip 1em

%{\bf Data Availability} Not applicable.
%\vskip 1em

%{\bf Declarations}

%{\bf Conflict of interest}  The authors have not disclosed any competing interests.

\section*{ }
\bibliographystyle{acm}

\end{document}